\documentclass[12pt]{amsart}
\usepackage[utf8]{inputenc}
\usepackage[T1]{fontenc}
\usepackage[english]{babel}

\usepackage{amsthm}
\usepackage{amsmath}
\usepackage{amssymb}
\usepackage{mathrsfs}
\usepackage{soul}
\usepackage{mathabx}

\usepackage[hmargin=2.5cm,vmargin=3.5cm]{geometry}

\newcommand{\N}{\mathbb{N}}
\newcommand{\Z}{\mathbb{Z}}

\newcommand{\C}{\mathbb{C}}

\newcommand{\E}{\mathbb{E}}

\newcommand{\Mn}{M_n}
\newcommand{\M}{\mathcal{M}}
\newcommand{\im}{\text{Im}}
\newcommand{\F}{\mathcal{F}}

\newcommand{\I}{\mathcal{I}}
\newcommand{\tr}{\text{tr}}

\newtheorem{theorem}{Theorem}[section]
\newtheorem{definition}[theorem]{Definition}

\newtheorem{rmq}[theorem]{Remark}

\newtheorem{lemme}[theorem]{Lemma}
\newtheorem{corollaire}[theorem]{Corollary}

\begin{document}

\title[Absolutely dilatable Schur multipliers]
{Unital positive Schur multipliers on $S_n^p$ with a completely isometric dilation}

\author[C. Duquet]{Charles Duquet}
\email{charles.duquet@univ-fcomte.fr}
\address{Laboratoire de Math\'ematiques de Besan\c con,  Universit\'e de Franche-Comt\'e,
16 Route de Gray, 25000 Besan\c{c}on Cedex, FRANCE}

\date{\today}

\maketitle

\begin{Large}Abstract:\end{Large} Let $1<p\not=2<\infty$ and let $S^p_n$ be the 
associated Schatten von Neumann class over $n\times n$ matrices. We prove 
new characterizations of unital positive Schur multipliers $S^p_n\to S^p_n$ 
which can be dilated into an invertible complete isometry acting on a non-commutative
$L^p$-space. Then we investigate the infinite dimensional case.



\bigskip

\section{Introduction}

A famous theorem of Akcoglu \cite{Akcoglu} asserts that positive contractions on classical $L^p$-spaces, $1<p<\infty$, admit an isometric dilation, as follows:
for any measure space $(\Sigma,\mu)$, any $1<p<\infty$ and  
any positive contraction $T:L^p(\Sigma)\to L^p(\Sigma)$, 
there exist a measure space $(\Sigma',\mu')$, 
two contractions $J:L^p(\Sigma)\to L^p(\Sigma')$ and 
$Q:L^p(\Sigma')\to L^p(\Sigma)$ and an invertible isometry $U:L^p(\Sigma')\to L^p(\Sigma')$ such that  $T^k=QU^kJ$ for all integer $k\geq 0$.

The paper deals with similar dilation properties in the framework of noncommutative $L^p$-spaces
associated with semi-finite von Neumann algebras.
More specifically, we will be interested in dilation properties of Schur multipliers. Before presenting our results, we need a few preliminaries and some background.

 In the sequel we call tracial von Neumann algebra any pair $(N,\tau)$, 
 where $N$ is a (semi-finite) von Neumann algebra and $\tau$ is a normal semi-finite faithful trace (n.s.f trace in short) on $N$. 
 Assume that 
 $N$ is acting on some Hilbert space $H$.
 We let $L_0(N)$ denote the $*$-algebra of all $\tau$-measurable operators acting on $H$.
Any $x\in L_0(N)$ admits a 
(unique) polar decomposition $x=u|x|$. 
For $1\leq p<+\infty$, the noncommutative $L^p$-space associated with
$(N,\tau)$ is defined as 
$$
L^p(N):=\lbrace x\in L_0(N)\text{ : } \tau (|x|^p)<+\infty\rbrace.
$$
This is a Banach space for the norm 
$\Vert x\Vert_p=\bigl(\tau (|x|^p)\bigr)^{\frac{1}{p}}$.
We further set $L^\infty(N)=N$.

Let $1\leq p<\infty$ and $1<q\leq\infty$ such that 
$\frac{1}{p}+\frac{1}{q}=1$.
For any $x\in L^p(N)$ and $y\in L^q(N)$, we have
$$
xy\in L^1(N) \text{ and } |\tau(xy)|\leq \|x\|_p\|y\|_q.
$$
Furthermore the duality bracket $(x,y)\mapsto \tau(xy)$ between $L^p(N)$ and $L^q(N)$ yields an isometric identification
$L^p(N)^*=L^q(N)$. 
We refer to \cite[Section 4]{Hiai} for details.

Let $I$ be an index set. If $N=B(l^2(I))$ is equipped with the
usual trace, then the space $L^p(N)$ identifies with the Schatten von Neumann class $S_I^p$ for all $1\leq p<\infty$. 

For any $n\geq 1$ and any
$1\leq p\leq \infty$, we identify $L^p(M_n(N))$ 
with $S^p_n\otimes L^p(N)$ in the usual way. 
Let $T: L^p(M)\to L^p(N)$ be a linear map, where $M$ and $N$ are both tracial von Neumann algebras. For any $n\geq 1$, let $T_n: L^p(M_n(M))\to L^p(M_n(N))$ be defined by
$$
T_n([a_{ij}])=[T(a_{ij})],\qquad  [a_{ij}]\in L^p(M_n(M)).
$$
We say that $T$ is a complete contraction if for all $n\geq 1$, $\|T_n\|_{L^p(M_n(M))\to L^p(M_n(N))}\leq 1$. If $T^*:L^q(N)\to L^q(M)$ denotes the adjoint of $T$ and if $T$ is a complete contraction, then $T^*$ is a complete contraction. 
This follows from the fact that for any $n\geq 1$, $(T^*)_n=(T_n)^*$.  In the same way, 
$T$ is called a complete isometry if for all $n\geq 1$, $T_n$ is an isometry. We further say that $T$ is completely positive if for all $n\geq 1$, $T_n$ is positive.

\begin{definition}\label{def comp p dilatable}
Let $(N,\tau)$ be a tracial von Neumann algebra and let $1\leq p<\infty$.
We say that an operator $T:L^p(N,\tau)\to L^p(N,\tau)$ is completely $p$-dilatable if there exist a tracial von Neumann algebra $(N',\tau')$, two complete contractions $J:L^p(N,\tau)\to L^p(N',\tau')$
and $Q: L^p(N',\tau')\to L^p(N,\tau)$, and an invertible complete isometry $U:L^p(N',\tau')\to L^p(N',\tau')$, such that for all $k\geq 0$, $T^k=QU^kJ$.
\begin{align*}
\begin{array}{ccc}
L^p(N',\tau')&\overset{U^k}{\longrightarrow}&L^p(N',\tau')\\
J\uparrow& & \downarrow Q\\
L^p(N,\tau)&\overset{T^k}{\longrightarrow}&L^p(N,\tau)
\end{array}
\end{align*}
\end{definition}

We say that a positive operator 
$T: (N,\tau)\to (N',\tau')$ between two tracial von Neumann algebras is trace preserving if for all $x\in N_+$, we have $\tau'(T(x))=\tau(x)$. The following is well-known (see e.g. \cite[Lemma 1.1]{JX}).

\begin{lemme}\label{lem superdilatation}
Let $(N,\tau), (N',\tau')$ be two tracial von Neumann algebras
and let $T:N\to N'$ be a positive trace preserving complete contraction. 
Then for all $1\leq p<\infty$, there exists a necessarily unique complete contraction 
$T_p: L^p(N,\tau)\to L^p(N',\tau')$ such that for all $x\in N\cap L^p(N,\tau)$, $T(x)=T_p(x)$.

If further $T$ is a one-to-one $\star$-homomorphism, then $T_p$ is 
a complete isometry.
\end{lemme}

Assume that $J: (N,\tau)\to (N',\tau')$ is a unital one-to-one
trace preserving and $w^*$-continuous $\star$-homomorphism. Let $J_1 :L^1(N,\tau)\to L^1(N',\tau')$
be induced by $J$, according to Lemma \ref{lem superdilatation}. It
is well-known that $J_1^* :
N'\to N$ is a conditional expectation (if we regard $N$ as a von Neumann subalgebra of $N'$, using $J$). In the sequel, $J_1^*$ is called the conditional expectation associated with $J$.

\begin{definition} \label{rmq superdilatation}
We say that an operator
$T:(N,\tau)\to (N,\tau)$ is absolutely dilatable if there exist a
tracial von Neumann algebra $(N',\tau')$, a unital one-to-one trace preserving and $w^*$-continuous  $\star$-homomorphism $J:(N,\tau)\to (N',\tau')$ and a trace preserving $\star$-automorphism $U:(N',\tau')\to (N',\tau')$ such that
\begin{equation}\label{DilPpty}
T^k=\E U^kJ,\qquad k\geq 0,
\end{equation}
where $\mathbb{E}:N'\to N$ is the conditional 
expectation associated with $J$.
\end{definition}

Any absolutely dilatable operator is positive and trace preserving. Indeed with the above notation, (\ref{DilPpty}) with $k=1$ yields $T=\mathbb{E}UJ$ and $J,U,\mathbb{E}$ are all positive
and trace preserving.

If $T$ satisfies Definition \ref{rmq superdilatation}, 
then applying Lemma 
\ref{lem superdilatation}
we obtain that for all $1\leq p<\infty$, 
$J$ (resp. $\mathbb{E}$) induces a
complete contraction
$L^p(N,\tau)\to L^p(N',\tau')$ (resp. $L^p(N',\tau')\to L^p(N,\tau)$)
and that $U$ induces an invertible complete isometry $L^p(N',\tau')\to L^p(N',\tau')$.
Moreover (\ref{DilPpty}) holds true on $L^p$-spaces. We therefore obtain the 
following lemma.

\begin{lemme}\label{lem absdila implique dila p}
If $T:(N,\tau)\to (N,\tau)$ is absolutely dilatable, then for every $1\leq p <\infty$,
$T$ induces a contraction $T_p:L^p(N,\tau)\to L^p(N,\tau)$
and $T_p$ is completely $p$-dilatable.
\end{lemme}

In Section \ref{section Finite Schur multipliers}, we consider dilation properties of unital positive Schur multipliers on matrices. Let $n\geq 1$ and denote by $\Mn$ the space of $n\times n$ matrices with complex coefficients. Let $T_M:\Mn\to \Mn$ be the Schur multiplier
defined by $T_M([a_{ij}]_{1\leq i,j\leq n})=
[m_{ij}a_{ij}]_{1\leq i,j\leq n}$, for some 
$M=(m_{ij})_{1\leq i,j\leq n}$  in $\Mn$. 
Assume that $T_M$ is unital and positive. We prove that if $T_M$ 
(regarded as acting on $S^p_n$)
is completely $p$-dilatable for some $1< p\neq 2<+\infty$, then there exist a tracial normalised von Neumann algebra $(N,\tau)$
and unitaries $v_1,\dots,v_n$ in $N$ 
such that for all $1\leq i,j\leq n$,
$$
m_{ij}=\tau( v_i^*v_j).
$$
By combining with results of \cite{HaagerupMusat}, we deduce that any Schur multiplier $T_M$ as before is completely $p$-dilatable for some $1\leq p\neq 2<+\infty$ if and only if $T_M$ 
is absolutely dilatable. 

In Section \ref{section Discrete Schur multipliers}, 
we extend the above result to Schur multipliers  
on $B(l^2(I))$ for any infinite index set $I$. In the last section, we consider 
simultaneous 
absolute dilations for finite 
commuting families of operators. We obtain that a finite
family of 
Schur multipliers is simultaneously absolutely dilatable 
if and only if each of these  Schur multipliers 
is absolutely dilatable.

\section{Finite Schur multipliers}\label{section Finite Schur multipliers}
In the following, we denote by $\Mn$ the space of $n\times n$ matrices with complex coefficients. Let $\tr_n$ be the usual trace on 
$M_n$. Unless otherwise stated we will consider $M_n$ as equipped with
the normalized trace 
$\tau_n:=\frac{1}{n}\tr_n$. The two traces $\tau_n$ and $\tr_n$ are needed in the following.

Let $M=(m_{ij})_{1\leq i,j\leq n}$ be an element of $\Mn$.
We denote by $T_M: M_n\to M_n$ the Schur multiplier associated with $M$, defined by
$$
T_{M}([b_{ij}]_{1\leq i,j\leq n})=[m_{ij}b_{ij}]_{1\leq i,j\leq n},\qquad  [b_{ij}]_{1\leq i,j\leq n}\in \Mn.
$$

Firstly, we recall some links between $T_M$ and $M$. 
In the next theorem,
the first point is in \cite[Theorem 3.7]{Paulsen}
and the second one is clear.
\begin{theorem}\label{th lien matrice et multiplicateur}
Let $M=(m_{ij})_{1\leq i,j\leq n}\in M_n$, then

\begin{enumerate}
\item $T_M$ is positive if and only if $T_M$ is completely 
positive if and only if $M$ is positive semi-definite.
\item $T_M$ is unital if and only if all diagonal entries of $M$ are equal to 1.
\end{enumerate}
\end{theorem}

We recall the following definition from 
\cite[Definition 4.1.7]{Takesaki1}.

\begin{definition}
Let $I$ be an index set.
A family $\lbrace w_{ij} : i,j\in I\rbrace $ of 
elements in a von Neumann algebra $N$ is called a set 
of matrix units if 
\begin{enumerate}
\item For all $i,j\in I$, $w_{ij}^*=w_{ji}$;
\item For all $i,j,k,l\in I$, $w_{ij}w_{kl}=\delta_{jk}w_{il}$;
\item $\sum_{i\in I} w_{ii}=1$ in the strong topology.
\end{enumerate}
\end{definition}

We give an alternative version of  \cite[Theorem 4.1.8]{Takesaki1}. 
In the sequel we let $e_{ij},\, 1\leq i,j\leq n$, 
denote the standard matrix units of $M_n$.

\begin{theorem}\label{bijection takasaki}
Let $(N,\tau)$ be a tracial von Neumann algebra.  
Suppose that $\lbrace w_{ij}:1\leq i,j\leq n\rbrace$ is a 
set of matrix units in $N$. Let $e=w_{11}$
and let $N_1=eNe$ be equipped with 
the restriction $\tau_1$ of $\tau$
to $N_1$. Then the following mapping is 
a trace preserving $\star$-isomorphism:
$$
\rho: N\to (\Mn, \tr_n)\otimes (N_1,\tau_1) 
\text {, } x\mapsto \sum_{1\leq i,j\leq n} e_{ij}\otimes w_{1i}xw_{j1}.
$$
\end{theorem}

\begin{proof}
See \cite[Lemma 2.2]{LMZ} and its proof.
\end{proof}

We refer to \cite[lemma 2.1]{HaagerupMusat} for the 
following lemma.

\begin{lemme}\label{lemme existence de u}
Let $P$ be a von Neumann algebra, let $n\geq 1$
be an integer and
let $\lbrace f_{ij}:\text{ }1\leq i,j\leq n\rbrace$
and $\lbrace g_{ij}:\text{ }1\leq i,j\leq n\rbrace$ 
be two sets of matrix units in $P$. 
Then there exists an unitary $u\in P$ such that:
$$
uf_{ij}u^*=g_{ij}, \qquad 1\leq i,j\leq n.
$$
\end{lemme}

We recall a classical duality
result and we give a proof for the sake of completeness. 

\begin{lemme}\label{lemme normalisateur}
Let $1<p,q<+\infty$ such that $\frac{1}{p}+\frac{1}{q}=1$, and let $(N,\tau)$ 
be a tracial von Neumann algebra.
For all $x\in L^p(N)\setminus\{0\}$, there exists a unique $y\in L^q(N)$
such that $\|y\|_q=1$ and $\tau(xy)=\|x\|_p$. (This element is called
the norming functional of $x$.)

Assume that $\|x\|_p=1$ and  write 
the polar decomposition of $x$ as $x=u|x|$. Then
the norming functional $y$ of $x$ satisfies
\begin{align} \label{forme y HB}
y=|x|^{p-1}u^*\quad\hbox{and}\quad |y|=u|x|^{p-1}u^*.
\end{align} 
Moreover, its polar decomposition is $y=u^*|y|$.
\end{lemme}

\begin{proof}
The first part of the statement means that 
$L^p(N)$ is smooth. This well-known fact follows e.g. 
from \cite[Corollary 5.2]{PX}.

For the second part, assume that $\|x\|_p=1$ and 
set $y=|x|^{p-1}u^*$.  Then we
have $y^*=u|x|^{p-1}$ and $yx=|x|^p$. 
This implies that $\tau(yx)=\|x\|_p^p=1$ and $\|y\|_{q}^p=\tau(|x|^p)=1$. 
Hence $y$ is the norming functional of $x$.

We will now prove that $|y|=u|x|^{p-1}u^*$ and $y=u^*|y|$. We have,
$$
(u|x|^{p-1}u^*)^2=u|x|^{p-1}u^*u|x|^{p-1}u^*=u|x|^{2(p-1)}u^*=y^*y.
$$
So we obtain $|y|=u|x|^{p-1}u^*$. We deduce that
$$
u^*|y|=u^*u|x|^{p-1}u^*=|x|^{p-1}u^*=y.
$$ 

For the polar decomposition,
it remains to show that $\im(|y|)^\perp \subset \ker u^*$.
We have $\im(|y|)^\perp=\ker(|y|)=\ker(y)$.
Let $h$ be an element of $\im(|y|)^\perp$, 
then $|x|^{p-1}u^*(h)=y(h)=0$, and so 
$u^*(h) \in \ker (|x|^{p-1})=\ker(|x|)$.
Moreover $u^*(h)\in \im(u^*)\subset \ker(|x|)^\perp$. Consequently,
$u^*(h)=0$, and hence $h$ belongs to $\ker(u^*)$.
\end{proof}

Let $H$ be a Hilbert space and
let $x$ be a closed densely defined operator on $H$.
We let $D(x)\subset H$ denote the domain of $x$. If $x$ is selfadjoint, then we let $s(x) : H\to H$ denote the support of $x$, which is equal to 
the orthogonal projection whose range is 
$\overline{\im(x)}$.

\begin{lemme}\label{lemme xy=0 y=0}
Let $(\M, \tau)$ be a tracial von Neumann algebra, let  $x\in L_0(\M)$ be a positive element and let $z\in s(x)\M s(x)$ such that $xzx=0$. Then $z=0$.
\end{lemme}

\begin{proof}
Changing $\M$ into $s(x)\M s(x)$, 
we may assume that $s(x)=1$. This means that ${\rm ker}(s)=\{0\}$. 
Let $z\in \M$ such that $xzx=0$. For any 
$\xi\in D(xzx)$, we have,
$$
xzx(\xi)=0,
$$
hence $zx(\xi)=0$.  By density of $D(xzx)$ in $D(zx)$, 
this implies that $zx=0$. We deduce that
$xz^*=0$. Then arguing as above,
we obtain that $z^*=0$, and hence $z=0$.
\end{proof}

The next definition is from \cite[Definition 1.3]{HaagerupMusat}.

\begin{definition} 
Let $(N,\tau)$ be a tracial von Neumann algebra and assume 
that $\tau$ is normalised.
We say that $T:(N,\tau)\to (N,\tau)$ is factorisable if there exist another tracial normalised von Neumann algebra $(N',\tau')$ and two unital trace-preserving 
$\star$-homomorphisms $\widetilde{J}, J:(N,\tau)\to (N',\tau')$ 
such that $T=\widetilde{J}_1^*J$.
\end{definition}

By Theorem \ref{th lien matrice et multiplicateur}, a Schur multiplier $T_M$ is positive and unital if and only if $M$ is a positive semi-definite complex matrix having all diagonal entries equal to 1. 

The following two theorems are \cite[Proposition 2.8]{HaagerupMusat} and \cite[Theorem 4.4]{HaagerupMusat},
respectively.

\begin{theorem}\label{th equi facto uu}
Let $M=(m_{ij})_{1\leq i,j\leq n}$ be a positive semi-definite complex matrix having all diagonal entries equal to 1. Then the Schur multiplier $T_M$ associated with $M$ is factorisable if and only if there exist a tracial normalised von Neumann algebra $(N,\tau)$ and an $n$-tuple $(v_1,\dots,v_n)$ of
unitaries of $N$ such that for all $1\leq i,j\leq n$:
$$
m_{ij}=\tau(v_i^*v_j).
$$
\end{theorem}

\begin{theorem}\label{th equi dila facto}
Let $(N,\tau)$ be a normalized
tracial von Neumann algebra.
Then $T:N\to N$ is factorisable if and only if  $T$  is absolutely dilatable where the von Neumann algebra $N'$ is normalised. 
\end{theorem}

The main result of this paper is
the equivalence $1\Leftrightarrow 5$ of the next theorem.
The other equivalences in this theorem
are already known.

\begin{theorem}\label{th principal 2}
Let $M=(m_{ij})_{1\leq i,j\leq n}$ be a positive semi-definite complex matrix having all diagonal entries equal to 1. The following assertions are equivalent:
\begin{enumerate}
\item there exists $1< p\neq 2 < +\infty$ such that $T_M:
S^p_n\to S^p_n$ is completely $p$-dilatable;
\item for all $1\leq p < +\infty$, 
$T_M:
S^p_n\to S^p_n$ is completely $p$-dilatable;
\item $T_M$ is absolutely dilatable;
\item $T_M$ is factorisable;
\item there exist  a tracial normalised von Neumann algebra $(N,\tau_{N})$ and an $n$-tuple $(v_1,\dots,v_n)$ of
unitaries of $N$ such that for all $1\leq i,j\leq n$:
$$
m_{ij}=\tau(v_i^*v_j).
$$
\end{enumerate}
\end{theorem}

\begin{proof}
Using Theorem \ref{th equi dila facto} and Theorem 
\ref{th equi facto uu}, we obtain the equivalences 
$3\Leftrightarrow 4 \Leftrightarrow 5$. 
The implications $3\Rightarrow 2 \Rightarrow 1$ follow 
from Lemma \ref{lem absdila implique dila p}. It 
therefore remains to prove that 1 implies 5. 

Let $1<p\neq 2<+\infty$ and 
suppose that $T_M:
S^p_n\to S^p_n$ is completely $p$-dilatable.
Let $1<q<\infty$ such that $\frac{1}{p}+\frac{1}{q}=1$.
There exist a tracial von Neumann algebra $(N,\tau_N)$, two complete contractions  $J:S_n^p\to L^p(N)$ and $Q: L^p(N)\to S_n^p$ and an
invertible complete  isometry $U:L^p(N)\to L^p(N)$ such that for all $k\in \N$:
\begin{align}\label{formule tm=quj}
T_M^k=QU^kJ.
\end{align}

We set $V=UJ:S_n^p\to L^p(N)$ and $J':=Q^*:S_n^q\to L^q(N)$.
Applying \eqref{formule tm=quj} with $k=1$, we obtain 
that 
\begin{equation}\label{Product}
T_M=(J')^*V. 
\end{equation}
Next applying \eqref{formule tm=quj} with
$k=0$, we obtain that $QJ=Id$. Since $Q,J$ are  complete contractions, this implies that $J$ is a complete isometry.
Likewise, $J^*Q^*=Id$, hence $Q^*$ is a complete isometry. Consequently $V$ and $J'$ are both complete isometries.

Firstly, we will show that we 
can suppose that $J'$ and $V$ are positive. 
Since $p\neq 2$, Yeadon's Theorem describing $L^p$-isometries
applies to $J'$ and $V$. More precisely, by \cite[Proposition 3.2]{JRS}, there exist partial isometries
$W_1,W_2\in N$, positive operators $B_1,B_2$
affiliated with $N$ and $\star$-homomorphisms  
$J_1: \Mn\to N$ and 
$J_2:\Mn \to N$ such that for all $x\in \Mn$, and $i=1,2$,
\begin{align*} 
& B_i \text{ commute with } J_i(x),\\
&\tau_n(x)=\tau_N(B_1^pJ_1(x))=\tau_N(B_2^q J_2(x)),\\
&V(x)=W_1B_1J_1(x) \text{ and }J'(x)=W_2B_2J_2(x),\\
& W_i^*W_i=s(B_i)=J_i(1).
\end{align*}
Note that since the domain space of $V$ has a finite trace,
we  actually have
that $B_1=|V(1)|$ belongs to $L^p(N)$ and 
the polar decomposition of $V(1)$ is $W_1B_1$. 
Likewise $B_2=|J'(1)|$ belongs to $L^q(N)$ 
and the polar decomposition of $J'(1)$ is $W_2B_2$.

It follows from (\ref{Product}) that
for all $x\in S^p_n$ and $y\in S^q_n$, we have
\begin{align}\label{TraceIdentity}
\tau_n(T_M(x)y)=\tau_N(V(x) J'(y)).
\end{align}
 
By Theorem \ref{th lien matrice et multiplicateur}, the operator $T_M$ is unital. Hence if we take $x=y=1$ in the last formula, we obtain
$$
1=\tau_n(1)=\langle W_1B_1,W_2B_2\rangle.
$$
In addition, we have $\|W_1B_1\|_p=\|V(1)\|_p=1$
and $\|W_2B_2\|_q=\|J'(1)\|_q=1$. According to
Lemma \ref{lemme normalisateur}, we deduce that
$W_2B_2=B_1^{p-1}W_1^*$ and $B_2=W_1B_1^{p-1}W_1^*$. Moreover
by the uniqueness of the polar decomposition of $J'(1)$, 
which is both $W_2B_2$ and $W_1^*B_2$, we obtain that $W_1=W_2^*$. 
We define $\Tilde{V}: S_n^p\to L^p(N)$ and $\Tilde{J'}: S_n^q\to L^q(N)$ by setting 
$$
\Tilde{V}(x):=B_1J_1(x) \quad\text{ and }\quad \Tilde{J'}(y):=B_1^{p-1}W_1^*J_2(y)W_1.
$$
for all $x\in S^p_n$ and $y\in S^q_n$.
Note that we both have $\Tilde{V}(.)=W_1^*V(.)$ and 
$V(.)=W_1\Tilde{V}(.)$. This implies that $\|\Tilde{V}(x)\|_p\leq
\|V(x)\|_p$ and $\|V(x)\|_p\leq\|\Tilde{V}(x)\|_p$ for all $x\in S_n^p$.
Then we have $\|V(x)\|_p=\|\Tilde{V}(x)\|_p$ for all $x\in S^p_n$. 
Hence $\Tilde{V}$ is an isometry. We may show as well that 
$\Tilde{V}$ is a complete isometry. Since $J_1$ and $B_1$ commute, 
the operator $\Tilde{V}$ is positive.

Let us  prove that $\Tilde{J'}$ is also a positive complete
isometry. We first note that $B_1^{p-1}$ and $W_1^*J_2(y)W_1$ 
commute for all $y\in \Mn$. Indeed, given any $y\in \Mn$, we have
\begin{align*}
B_1^{p-1}W_1^*J_2(y)W_1&=W_2B_2J_2(y)W_1=W_2J_2(y)B_2W_1=W_2J_2(y)W_2^*W_2B_2W_1\\
&=W_2J_2(y)W_2^*B_1^{p-1}W_1^*W_1=W_2J_2(y)W_2^*B_1^{p-1}=W_1^*J_2(y)W_1B_1^{p-1},
\end{align*}
which proves the commutation property.
Since $B_1^{p-1}$ and  $J_2(y)$ are positive, if $y$ is positive, 
the latter implies that $\Tilde{J'}$ is positive. 

We  now prove that $W_1^*J_2(.)W_1$ is multiplicative. 
Let $x,y\in \Mn$. Since $W_1^*=W_2$, we have
$$
W_1^*J_2(xy)W_1=W_1^*J_2(x)J_2(y)W_1=W_1^*J_2(x)W_2^*W_2J_2(y)W_1
=W_1^*J_2(x)W_1W_1^*J_2(y)W_1,
$$
which proves the result. Consequently,
$W_1^*J_2(.)W_1: M_n\to N$ is a $\star$-homomorphism.

In addition we have, for all $y\in \Mn$,
 \begin{align*}
\tau_n(y)&
=\tau_N\left(B_2^qJ_2(y)\right)\\
& =\tau_N\left(\left(W_1(B_1^{p-1})W_1^*\right)^qJ_2(y)\right)\\
&=\tau_N\left(W_1(B_1^{p-1)})^q W_1^*J_2(y)\right)\\
&=\tau_N\left((B_1^{p-1})^q W_1^*J_2(y)W_1\right).
\end{align*}
We also have
$$
 W_1^*J_2(1)W_1=W_1^*W_2^*W_2W_1=W_1^*W_1W_1^*W_1=s(B_1)s(B_1)=s(B_1).
$$
By the characterisation of  positive complete isometries (see \cite[Theorem 3.1]{JRS} and \cite[Proposition 3.2]{JRS}), we
deduce that 
the operator $\Tilde{J'}$ is an positive complete isometry.

We now check that $T_M=\Tilde{J'}^*\Tilde{V}$. Let $x,y\in \Mn$,
 \begin{align*}
\tau_n(\Tilde{J'}^*\Tilde{V}(x)y)&=
\tau_N(\Tilde{V}(x)\Tilde{J'}(y))=\tau_N(B_1J_1(x)B_1^{p-1}W_1^*J_2(y)W_1)\\
&=\tau_N(W_1B_1J_1(x)W_2B_2J_2(y))=
\tau_N(V(x)J'(y))\\
&=\tau_n(T_M(x)y),
\end{align*}
by (\ref{TraceIdentity}).

It follows from above that we 
may now suppose that $J'$ and $V$ are positive. 
Once again we use the Yeadon decomposition 
for the complete isometries of $J'$ and $V$. 
In the positive case, we can remove the partial isometry in 
the Yeadon decomposition. We therefore obtain that
there exist two positive operators $B_1\in L^p(N)$
and $B_2\in L^q(N)$ as well as two $\star$-homomorphisms  
$J_1: \Mn\to N$ and $J_2:\Mn \to N$ 
such that for all $x\in \Mn$, and $i=1,2$,
\begin{align} 
& B_i \text{ commutes with } J_i(x), 
\label{propriete de Yeadon commutent}\\
&\tau_n(x)=\tau_N(B_1^pJ_1(x))=\tau_N(B_2^q J_2(x)),
\label{propriete de Yeadon trace}\\
&V(x)=B_1J_1(x) \text{ and }J'(x)=B_2J_2(x), 
\label{propriete de Yeadon ecriture}\\
& s(B_i)=J_i(1) \label{support}.
\end{align}
Argunig as above, we obtain that 
$1=\tau_N(B_1B_2)$ from which we deduce that 
$$
B_2=B_1^{p-1}.
$$

We set $B:=B_1^p=B_2^q$, this is an 
element of $L^1(N)$. Thanks to 
\eqref{propriete de Yeadon commutent}, 
$B$ commutes with $J_1(x)$ and $J_2(x)$ for 
all $x\in \Mn$.

Consider the projection $v:=J_1(1)$. Then we have 
$$
v:=J_1(1)=s(B_1)=s(B)=s(B_2)=J_2(1).
$$
We define $\M:=vNv$, and we equip this von Neumann
algebra with the trace $\tau_\M=\tau_{N|_\M}$. 
The $\star$-homomorphisms $J_1,J_2$ are valued in $\M$.
Henceforth we consider $J_1: M_n\to \M$ and $J_2: M_n\to \M$,
they are now unital $\star$-homomorphisms. 
We set
$$
f_{ij}=J_1(e_{ij})\quad\hbox{and}\quad
g_{ij}=J_1(e_{ij}),\qquad 1\leq i,j\leq n.
$$
Then $\lbrace f_{ij}:\text{ }1\leq i,j\leq n\rbrace$ and 
$\lbrace g_{ij}:\text{ }1\leq i,j\leq n\rbrace$ 
are two sets of matrix units.
By Lemma \ref{lemme existence de u}, 
there exists a unitary $u_1\in \M$ such that
\begin{align} \label{eq u}
u_1f_{ij}u_1^*=g_{ij},\qquad 1\leq i,j\leq n.
\end{align}
We set $w_i:=g_{i1}=J_2(e_{i1})$ for all $1\leq i\leq n$. 
We set $N'=w_1\M w_1$ and we equip it with
$\tau_{N'} : =\tau_{\M|_{N'}}$.
We define
$$
\rho: \M\to (\Mn, \tr_n)\otimes (N',\tau_{N'}) 
\text {, } \ x\mapsto \sum_{i,j=1}^n e_{ij}\otimes w_i^*xw_j.
$$
According to Theorem \ref{bijection takasaki}, $\rho$ is a trace
preserving $\star$-isomorphism. 
By Lemma \ref{lem superdilatation} we extend $\rho$ to 
an isometry from $L^p(\M)$ into $L^p(\Mn\otimes N')
=S_n^p\otimes L^p(N')$, that we still denote by  $\rho$.

We note that for all $x\in \Mn$,
$$
\rho(J_2(x))=x\otimes 1_{N'}.
$$
Indeed,  for all $x=\sum_{i,j=1}^n \alpha_{ij}e_{ij}\in \Mn$ we have;
\begin{align*}
\rho(J_2(x))&=\sum_{i,j=1}^n \alpha_{ij}\rho(J_2(e_{ij}))\\
&=\sum_{i,j=1}^n \alpha_{ij}\rho(g_{ij})\\
&=\sum_{i,j=1}^n \alpha_{ij}\sum_{k,l=1}^n e_{kl}\otimes w_k^*g_{ij}w_l\\
&=\sum_{i,j=1}^n \alpha_{ij}\sum_{k,l=1}^n e_{kl}
\otimes g_{1k}g_{ij}g_{l1}\\
&=\sum_{i,j=1}^n \alpha_{ij} e_{ij}\otimes g_{11}\\
&=x\otimes 1_{N'}.
\end{align*}
Since $\rho$ is an unital $\star$-homomorphism, $u:=\rho(u_1)\in \Mn\otimes N'$ is also a unitary. Thanks to 
property \eqref{eq u}, we have $\rho(J_1(x))=\rho(u_1^*J_2(x)u_1)$, for all $x\in \Mn$. So we obtain that for all $x\in \Mn$,
$$
\rho(J_1(x))=u^*(x\otimes 1_{N'})u.
$$

Let us use the commutation of $B$ with $J_2$. 
For all $x\in \Mn$, we have
$$
\rho(B)(x\otimes 1_{N'})=\rho(BJ_2(x))=\rho(J_2(x) B)=(x\otimes 1_{N'})\rho(B).
$$
This implies that $\rho(B)=1\otimes b_{11}$,
for some $b_{11}\in L^p(N')$. 
Since $\tau_M(B)=1$ and $\rho$ is trace preserving, we have 
$$
\tau_{N'}(b_{11})=\dfrac{1}{n}.
$$

Let $z\in \Mn\otimes N'$ such that $\rho(B)^\frac{1}{2}z\rho(B)^\frac{1}{2}=0$. Since $\rho$ is a  bijective homomorphism, we have $B^\frac{1}{2}\rho^{-1}(z)B^\frac{1}{2}=0$. By Lemma \ref{lemme xy=0 y=0}, we obtain that $\rho^{-1}(z)=0$. Finally, we have $z=0$. We summarize this by writing that for all $z\in \Mn\otimes N'$,
\begin{align}\label{rhoByrhoB=0 donc y=0}
\bigl\{\rho(B)^\frac{1}{2}z\rho(B)^\frac{1}{2}=0\bigr\}\,\Rightarrow\, z=0.
\end{align}

We write $u=\sum_{i,j=1}^n e_{ij}\otimes u_{ij}$, with 
$u_{ij}\in N'$. Let 
$x\in S^p_n$ and $y\in S^q_n$.
Owing to the commutation of $B$ and $J_2(x)$, 
we have:
\begin{align*}
\langle B_2J_2(y),B_1J_1(x)\rangle=\tau_N(B_1^{p-1}J_2(y)B_1J_1(x))=\tau_N(BJ_2(y)J_1(x)).
\end{align*}
By (\ref{TraceIdentity}), this implies that
\begin{align}\label{trace Tm en BJ2J1}
\tau_n(yT_M(x))=\tau_N(BJ_2(y)J_1(x)).
\end{align}

We take $y=e_{kk}$ and $x=e_{ll}$ with $1\leq k\neq l\leq n$. We deduce from above  that
\begin{align*}
0&=m_{ll}\tau_n(e_{kk}e_{ll})=\tau_n(e_{kk}T_M(e_{ll}))=\tau_N(BJ_2(e_{kk})J_1(e_{ll}))
\end{align*}
Since $BJ_2(e_{kk})J_1(e_{ll})\in \M$ and $J_2(e_{kk})^2=J_2(e_{kk})$, this implies
\begin{align*}
0&=\tau_\M(BJ_2(e_{kk})J_1(e_{ll}))\\
&=\tau_\M(BJ_2(e_{kk})J_2(e_{kk})J_1(e_{ll}))\\
&=\tau_\M(J_2(e_{kk})BJ_2(e_{kk})J_1(e_{ll}))\\
&=\tau_\M(BJ_2(e_{kk})J_1(e_{ll})J_2(e_{kk}))\\
&=\tau_{\Mn\otimes N'}(\rho(BJ_2(e_{kk})J_1(e_{ll})J_2(e_{kk})))\\
&=\tau_{\Mn\otimes N'}(\rho(B)\rho(J_2(e_{kk}))\rho(J_1(e_{ll}))\rho(J_2(e_{kk}))).
\end{align*}
Replacing $\rho(J_2(e_{kk}))$ by $e_{kk}\otimes 1_{N'}$, $\rho(J_1(e_{ll}))$ by $u^*(e_{ll}\otimes 1_{N'})u $ and $u^*$ by $\sum_{i,j=1}^n e_{ji}\otimes u^*_{ij}$, respectively, we deduce that
\begin{align*}
0&=\tau_{\Mn\otimes N'}\left(\rho(B)(e_{kk}\otimes 1_{N'})\sum_{i,j=1}^n (e_{ji}\otimes u_{ij}^*)(e_{ll}\otimes 1_{N'})u(e_{kk}\otimes 1_{N'})\right)\\
&=\sum_{i,j=1}^n\tau_{\Mn\otimes N'}\left(\rho(B)(e_{kk}e_{ji}e_{ll}\otimes u_{ij}^*)u(e_{kk}\otimes 1_{N'})\right)\\
&=\tau_{\Mn\otimes N'}\left(\rho(B)(e_{kl}\otimes u_{lk}^*)u(e_{kk}\otimes 1_{N'})\right).
\end{align*}
Next, replacing $u$ by $\sum_{i,j=1}^n e_{ij}\otimes u_{ij}$, we deduce that
\begin{align*}
0&=\tau_{\Mn\otimes N'}\left(\rho(B)(e_{kl}\otimes u_{lk}^*)\sum_{i,j=1}^n (e_{ij}\otimes u_{ij})(e_{kk}\otimes 1_{N'})\right)\\
&=\sum_{i,j=1}^n\tau_{\Mn\otimes N'}\left(\rho(B)(e_{kl}e_{ij}e_{kk}\otimes u_{lk}^*u_{ij})\right)\\
&=\tau_{\Mn\otimes N'}\left(\rho(B)(e_{kk}\otimes u_{lk}^*u^*_{lk})\right)\\
&=\tau_{\Mn\otimes N'}\left(\rho(B)(e_{kk}\otimes u_{lk}^*)(e_{kk}\otimes u_{lk})\right)\\
&=\tau_{\Mn\otimes N'}\left(\rho(B)^\frac{1}{2}(e_{kk}\otimes u_{lk}^*)(e_{kk}\otimes u_{lk})\rho(B)^\frac{1}{2}\right).
\end{align*}
Since $\rho(B)^\frac{1}{2}(e_{kk}\otimes u_{lk}^*)(e_{kk}\otimes u_{lk})\rho(B)^\frac{1}{2}$ is positive and the trace $\tau_{M_n\otimes N'}$ is faithful, 
we deduce that $\rho(B)^\frac{1}{2}(e_{kk}\otimes u_{lk}^*)(e_{kk}\otimes u_{lk})\rho(B)^\frac{1}{2}=0$ 
for all $k\neq l$. According to \eqref{rhoByrhoB=0 donc y=0}, this implies that 
$(e_{kk}\otimes u_{lk}^*)(e_{kk}\otimes u_{lk})=0$, 
hence 
$e_{kk}\otimes u_{lk}=0$. Finally for all  $1\leq k\neq l \leq n$, we have $u_{lk}=0$.
Therefore, the element $u$ has the form $\sum_{i=1}^ne_{ii}\otimes u_{ii}$. 
We set $u_i:=u_{ii}$, for $1\leq i\leq n$.
Then each $u_i$ is a unitary.

We will use the commutation property of $B$ and $J_1$. For all $x\in \Mn$, we have
$$
\rho(B)u^*(x\otimes 1_{N'})u=\rho(BJ_1(x))=\rho(J_1(x) B)=u^*(x\otimes 1_{N'})u\rho(B).
$$
Consider any $1\leq k,l\leq n$. We have
\begin{align*}
\rho(B)u^*(e_{kl}\otimes 1_{N'})u&=\rho(B)\left(\sum_{i=1}^ne_{ii}\otimes u_{i}^*\right)(e_{kl}\otimes 1_{N'})\left(\sum_{j=1}^ne_{jj}\otimes u_{j}\right)\\
&=\sum_{i,j=1}^n\rho(B)e_{ii}e_{kl}e_{jj}\otimes u_{i}^*u_{j}=\rho(B)e_{kl}\otimes u_{k}^*u_{l}\\
&=e_{kl}\otimes b_{11}u_{k}^*u_{l}.
\end{align*}
In the same way,
$$
u^*(x\otimes 1_{N'})u\rho(B)=e_{kl}\otimes u_{k}^*u_{l}b_{11}.
$$
We therefore deduce from above that
$$
u_{k}^*u_{l}b_{11}=b_{11}u_{k}^*u_{l},
\qquad 1\leq k,l\leq n.
$$
Thus for all $1\leq k,l\leq n$, we have $u_{l}b_{11}u_{l}^*=u_{k}b_{11}u_{k}^*$.
Hence $u_{k}b_{11}u_{k}^*
=u_1 b_{11}u_1^*$ for all $1\leq k\leq n$.

We will now relate $m_{ij}$, $b_{11}$ and $u_k$. 
For any $1\leq k,l\leq n$, we have
\begin{align*}
m_{kl}&=m_{kl}n\tau_n(e_{lk}e_{kl})=n\tau_n(e_{lk}T_M(e_{kl})).
\end{align*}
By \eqref{trace Tm en BJ2J1}, this implies
\begin{align*}
m_{kl}&=n\tau_N(BJ_2(e_{lk})J_1(e_{kl}))=n\tau_\M(BJ_2(e_{lk})J_1(e_{kl}))\\
&=n\tau_{\Mn\otimes N'}(\rho(B)\rho(J_2(e_{lk}))\rho(J_1(e_{kl})))\\
&=n\tau_{\Mn\otimes N'}(\rho(B)(e_{lk}\otimes 1_{N'})u^*(e_{kl}\otimes 1_{N'})u)\\
&=n\tau_{\Mn\otimes N'}\left(\rho(B)(e_{lk}\otimes 1_{N'})\left(\sum_{i=1}^n (e_{ii}\otimes u_i^*)\right)(e_{kl}\otimes 1_{N'})\left(\sum_{j=1}^n (e_{jj}\otimes u_j)\right)\right)\\
&=n\tau_{\Mn\otimes N'}(\rho(B)(e_{lk}\otimes 1_{N'})(e_{kl}\otimes u_k^* u_l))\\
&=n\tau_{\Mn\otimes N'}(\rho(B)(e_{ll}\otimes u_k^* u_l))\\
&=n\tau_{\Mn\otimes N'}(e_{ll}\otimes b_{11}u_k^* u_l))\\
&=n\tau_{ N'}(b_{11}u_k^* u_l)).\\
 \end{align*}

We set $b=nu_1b_{11}u_1^*$ and $v_i=u_iu_1^*$, for all $1\leq i\leq n$. By construction, $b\in L^1(N')$ and $b$ is positive. Moreover $\tau_{N'}(b)=1$. Furthermore, $v_1,\dots,v_n$ are unitaries. For any $1\leq i\leq n$, we have
$$
v_i^*bv_i=nu_1u_i^*u_1b_{11}u_1^*u_iu_1^*=
nu_1u_i^*u_ib_{11}u_i^*u_iu_1^*=nu_1b_{11}u_1^*=b.
$$
Thus, $b$ commutes with all the $v_i$. 
In addition, we obtain that for all $y\in N'$,
$$
b^\frac{1}{2}yb^\frac{1}{2}=0 \Rightarrow nu_1b_{11}^\frac{1}{2}u_1^*ynu_1b_{11}^\frac{1}{2}u_1^*=0 \Rightarrow b_{11}u_1^*yu_1b_{11}=0 \Rightarrow u_1^*yu_1=0 \Rightarrow y=0,
$$
by (\ref{rhoByrhoB=0 donc y=0}).
Finally for any 
$1\leq k,l\leq  n$, we have,
\begin{align*}
\tau_{N'}(bv_k^*v_l)=n\tau_{N'}(u_1b_{11}u_1^*u_1u_k^*u_lu_1^*)=n\tau_{N'}(b_{11}u_k^*u_l)=m_{kl}.
\end{align*}

Let us summarize the situation. 
We have obtained a tracial von Neumann algebra $(N',\tau')$, a positive $b\in L^1(N')$ with $\tau'(b)=1$, and unitaries 
$v_1,\dots,v_n\in N'$ such that 
$$
\forall\, 1\leq i,j\leq n,\ 
b v_i=v_i b \text{, }m_{ij}=\tau(b v_i^*v_j) 
\quad\text{ and }\quad 
\forall y\in N',\  b^\frac{1}{2}yb^\frac{1}{2}=0\Rightarrow y=0.
$$
Let $\Tilde{N}$ be the von Neumann algebra generated by $v_1,\ldots,v_n$ and let 
$\tau_{\Tilde{N}} : \Tilde{N}\to\mathbb C$
be defined by
$$
\tau_{\Tilde{N}}(x) = \tau'(bx).
$$
Then $\tau_{\Tilde{N}}$ is a normal state. 
Since $b$ commutes with all the $v_i$, it
commutes with all the elements of $\Tilde{N}$. 
Hence for all $x,y\in\Tilde{N}$,  we have
$$
\tau_{\Tilde{N}}(xy)=\tau'(bxy)=\tau'(xby)=\tau'(byx)=\tau_{\Tilde{N}}(yx).
$$
That is, $\tau_{\Tilde{N}}$ is a trace. Let $x\in \Tilde{N}_+$, such that $\tau_{\Tilde{N}}(x)=0$.
Then we have,
$$
0=\tau_{\Tilde{N}}(x)=\tau'(bx)=
\tau'(b^\frac{1}{2}xb^\frac{1}{2}).
$$ 
Since $\tau'$ is faithful, we obtain that 
$b^\frac{1}{2}xb^\frac{1}{2}=0$. By the 
above property of $b$, we deduce that $x=0$. Hence $\tau_{\Tilde{N}}$ is faithful. Thus 
$(\Tilde{N},\tau_{\Tilde{N}})$ is a 
tracial normalized von Neumann algebra. In addition  $(v_1,\dots,v_n)$ are unitaries of $\Tilde{N}$
and verify for all $1\leq i,j\leq n$,
$$
 m_{ij}=\tau(bv_i^*v_j)=\tau_{\Tilde{N}}(v_i^*v_j).
$$
This shows property 5 in Theorem \ref{th principal 2}.
\end{proof}

\begin{rmq}\label{rmq cas p=2 fini}
Consider $T_M : S^p_n\to S^p_n$ as in Theorem \ref{th principal 2} and 
for any $1\leq p<\infty$, let us say
that $T_M$ is completely positively $p$-dilatable if there exist
a tracial von Neumann algebra $(N,\tau)$, two completely positive and completely
contrative maps $J: S^p_n\to L^p(N)$ and $Q : L^p(N)\to S^p_n$, and an
invertible completely positive isometry $U:L^p(N)\to L^p(N)$ such that 
$T_M^k=QU^kJ$ for all $k\geq 0$. Then the five conditions of Theorem \ref{th principal 2} 
are also equivalent to:
\begin{itemize}
\item [1'.] there exists $1\leq p<\infty$ such that $T_M$ is completely positively $p$-dilatable.
\end{itemize}
Note that the case $p=2$ is admissible in this assertion. 

To prove this, it suffices to observe that if $T_M$ is completely positively $p$-dilatable
for some $1\leq p<\infty$, then (\ref{Product}) holds true for
some completely positive isometries $V\colon S^p_n\to L^p(N)$ and
$J'\colon S^q_n\to L^q(N)$ (here, $\frac{1}{p}+\frac{1}{q}=1$). Furthemore,
$V$ and $J'$ admit a Yeadon type factorization as
in (\ref{propriete de Yeadon commutent})-(\ref{propriete de Yeadon trace})-(\ref
{propriete de Yeadon ecriture})-(\ref{support}),
see \cite[Remark 5.2 and Theorem 4.2]{LMZl1contractive}. Therefore, 
the proof of Theorem \ref{th principal 2} shows as well that $1'\Rightarrow 5$.
Moreover the proof of Lemma \ref{lem absdila implique dila p} shows that $5\Rightarrow 1'$.
\end{rmq}

According to \cite{JLM},
there exists a completely positive contraction 
$u:S_n^p\to S_n^p$ which is not completely $p$-dilatable. 
The proofs of this result given in \cite{JLM}
do not provide any information on $n$. However we note that 
\cite[Example 3.2]{HaagerupMusat} provides a
unital completely positive Schur multiplier $T_M : M_4\to M_4$ 
which is not factorisable.
Applying Theorem \ref{th principal 2}, we  deduce that
for all $1<p\neq 2<+\infty$, $T_M:S^p_4\to S^p_4$
is not completely $p$-dilatable.

\section{Discrete Schur multipliers}\label{section Discrete Schur multipliers}

Let $I$ be an index set and let $\M_I$ denote 
the space of the $I\times I$ matrices with 
complex entries. We regard $B(l^2(I))\subset \M_I$ in the usual
way. Given any $M=(m_{ij})_{i,j\in I}\in\M_I$, 
the Schur multiplier on $B(l^2(I))$ associated with 
$M$ is the unbounded operator $T_M$ whose domain $D(T_M)$ is the space of all $A=(a_{ij})_{i,j\in I}$ 
in $B(l^2(I))$ such that $T_M(A) :=(m_{ij}a_{ij})_{i,j\in I}$ belongs to $B(l^2(I))$. 
If $D(T_M)=B(l^2(I))$, then $T_M: B(l^2(I))\to B(l^2(I))$ is a bounded Schur multiplier on $B(l^2(I))$.
In this case, for all $1\leq p<+\infty$, $T_M$ restricts to a bounded operator (still denoted by)
$T_M : S^p_I\to S^p_I$.

From now on, we assume that $T_M$ associated with $M\in \M_I$ is a bounded Schur multiplier 
on $B(l^2(I))$. We recall classical properties.
Since $T_M$ is positive (resp. completely positive) if and only if for all finite set $F\subset I$, $(T_M)_{|B(l^2(F))} : B(l^2(F))\to B(l^2(F))$ is positive (resp. completely positive), the following 
is a direct consequence of Theorem \ref{th lien matrice et multiplicateur}.

\begin{theorem}
\begin{enumerate}
\item []
\item $T_M$ is positive if and only if $T_M$ is completely positive if and only if  for all finite set $F\subset I$, the matrix $(m_{ij})_{i,j\in F}$ is positive semi-definite.
\item $T_M$ is unital if and only for all $i\in I$, $m_{ii}=1$.
\end{enumerate}
\end{theorem}

We will generalize Theorem \ref{th principal 2} to the present setting of
discrete Schur multipliers. In the following,  
we simply say $T_M$ is completely $p$-dilatable if 
$T_M : S^p_I\to S^p_I$ is completely $p$-dilatable.

\begin{theorem}\label{equivalence l2}
Let $T_M$ be a unital positive Schur multiplier. The
following assertions are equivalent:
\begin{enumerate}
\item there exists $1< p\neq 2 < +\infty$ such that $T_M$ is completely $p$-dilatable;
\item for all $1\leq p < +\infty$, $T_M$ is completely $p$-dilatable;
\item $T_M$ is absolutely dilatable;
\item there exist a tracial normalised von Neumann algebra $(N,\tau)$ and 
a family $(v_i)_{i\in I}$ of unitaries of $N$ 
such that for all $ i,j\in I$:
$$
m_{ij}=\tau(v_i^*v_j).
$$
\end{enumerate}
\end{theorem}

\begin{proof}
The implications $3\Rightarrow 2 \Rightarrow 1$  are clear. 
The implication $4 \Rightarrow 3$ is proved implicitly in \cite[Proof of Theorem 4.2]{A1}, see also \cite{DL}. It remains to show $1\Rightarrow 4$.

Let $1<p\neq 2<+\infty$
and suppose that $T_M:S_I^p\to S_I^p$ is a completely $p$-dilatable. 
There exist a tracial von Neumann algebra $(\M, \tau_\M)$, 
two complete contractions $J:S^p_I\to L^p(\M)$
and $Q: L^p(\M)\to S_I^p$, as well as 
an inverible complete isometry $U$ on
$L^p(\M)$, such that for all $k\geq 0$, $T^k=QU^kJ$. 
Let $1<q<\infty$ such that $\frac{1}{p}+\frac{1}{q}=1$ and set
$J':=Q^*: S_I^q\to L^q(\M)$. By the equality $T^k=QU^kJ$ for 
$k=0$, we see (as in the proof of Theorem 
\ref{th principal 2}) that
$J$ and $J'$ are two complete isometries.

For any finite subset $F\subset I$, we 
let $T_{M, F}: S_ F^p\to S_ F^p$ 
denote the restriction of $T_M$, that is,
$T_{M, F}((a_{ij})_{i,j\in  F})=(m_{ij}a_{ij})_{i,j\in  F}$. 
We consider $J_ F=J_{|S_ F^p}$ and $J'_ F=J'_{|S_ F^q}$. It is clear that on 
$S_ F^p$, we have $T_{M, F}^k=(J'_F)^*U^kJ_ F$ for all $k\geq 0$. 
Hence $T_{M, F}$ is completely $p$-dilatable.

In the finite dimensional case, we have Theorem \ref{th principal 2} at our disposal. 
We use it and we obtain that there exist a von Neumann algebra $\M_ F$ 
equipped with a normal faithful normalized trace $\tau_{\M_ F}$ and 
unitaries $(d_{i, F})_{i\in F}$ of $\M_ F$ such that for all $i,j\in F$,
$$
m_{ij}=\tau_{\M_ F}(d^*_{i, F}d_{j, F}).
$$

We recall the following ultraproduct construction, 
see \cite[Section 11.5]{Pisiertensor} for details. 
Let $\F$ be a non trivial
ultrafilter on the index set $\left\lbrace F\subset I \text{ finite}\right\rbrace$.
Let $B$ be the C$^*$-algebra defined by 
$$
B=\left\lbrace x=(x_F)\in \prod\limits_{ F\subset I\text{ finite }}\M_ F\text{ }|\sup\limits_{ F\subset I\text{ finite}}\|x_ F\|_{{\mathcal M}_F}<+\infty\right\rbrace.
$$
Let $f_\F\in B^*$ be the tracial state defined 
for all $x=(x_F)\in B$ by
$$
f_\F(x)=\lim_\F \tau_{\M_F}(x_F).
$$
Let $H_\F$ be the Hilbert space associated with the tracial state $f_\F$ in the GNS construction. Let $L:B\to B(H_\F)$ be the $\star$-homomorphism induced by left multiplication. 
Consider the ideal $\I_\F$ of $B$ defined by
$$
\I_\F:=\left\lbrace x\in B\text{ }| f_\F(x^*x)=0\right\rbrace.
$$
Since $f_\F$ is a trace, we have $\I_\F=\ker(L)$. Let 
$B/ \I_\F$ denote the resulting
quotient C$^*$-algebra and let 
$q:B\to B/ \I_\F$ be the quotient map. 
Then  we have a one-to-one $\star$-homomorphism 
$L_\F: B/\I_\F\to B(H_\F)$ and a 
faithful normalized trace $\tau_\F: B/\I_\F\to\mathbb C$ such that
$$
L_\F\circ q=L\quad \text{ and }\quad \tau_\F \circ q=f_\F.
$$
A remarkable result is that 
$$
M_\F:=L_\F(B/\I_\F)\subset B(H_\F)
$$
is
a von Neumann algebra and $\rho_\F:=\tau_\F\circ L_\F^{-1}
: M_\F\to\mathbb C$ is normal. Thus, 
$(M_\F,\rho_\F)$ is a tracial normalised von Neumann algebra.

For  all $i\in I$, we define
$d_i:=L_\F(q( (d_{i,F})_{F\subset I\text{ finite}}))=L( (d_{i,F})_{F\subset I\text{ finite}})$. The 
$d_i$ are well-defined because all $d_{i,F}$ are unitaries, hence
$\|d_{i,F}\|=1$. Since $L$ is unital, $d_i$ 
is a unitary for all $i\in I$.
It remains to prove the formula $m_{ij}=\rho_\F( d_i^*d_j)$. 
Given any $i,j\in I$, we have
$$
d_i^*d_j = L_\F\left(q\left( (d_{i,F}^*d_{j,F})_F \right)\right),
$$
hence
\begin{align*}
\rho_\F(d_i^*d_j)&=
\tau_\F\left(q\left( (d_{i,F}^*d_{j,F})_F \right)\right)\\
&=f_\F\left( (d_{i,F}^*d_{j,F})_F\right)\\
&=\lim_\F \tau_{\M_F}\left( d_{i,F}^* d_{j,F}\right)\\
&= m_{ij}.
\end{align*}
\end{proof}

The equivalence $3 \Leftrightarrow 4$ in Theorem
\ref{equivalence l2} provides a new proof
of \cite[Corollary 7.2]{DL}. We refer to the latter paper 
for similar results in the non discrete case.

\begin{rmq}
Following \cite[page 4367]{eric}, consider
the normal faithful state on $B(l^2)$
with density equal to
the diagonal operator $D=\sum\limits_{i\geq 1} \lambda_i 
e_i\otimes e_i$, where $(e_i)_{i\geq 1}$ is the
canonical basis of $l^2$, $\lambda_i>0$ 
for all $i\geq 1$ 
and $\sum\limits_{i\geq 1} \lambda_i=1$. 

We say that $T_M$ is $(D,D)$-factorisable if $T_M$ is 
factorisable in the sense of \cite[Definition 1.3]{HaagerupMusat}. 
According to \cite[Theorem 4.4]{HaagerupMusat}, the equivalent
conditions of Theorem \ref{equivalence l2} in the case
$I=\mathbb N$ are 
also equivalent to ``$T_M$ is $(D,D)$-factorisable".
\end{rmq}

\begin{rmq} As in Remark \ref{rmq cas p=2 fini}, 
we can observe that the equivalent
conditions of Theorem \ref{equivalence l2} are 
also equivalent to:
\begin{itemize}
\item [1'] there exists $1\leq p<\infty$ such that $T_M$ is completely positively $p$-dilatable.
\end{itemize}

\end{rmq}

\section{Multivariable case}\label{section Multivariable case}
In this last section, we introduce the notion of simultaneous absolute dilation for a commuting finite family of operators and consider the special case of Schur multipliers. Note
that all bounded Schur multipliers 
are commuting.

\begin{definition}

\    
\begin{itemize}
\item [(a)]
Let $(N,\tau)$ be a tracial von Neumann algebra.
We say that a commuting family $(T_1,\ldots,T_n)$ of 
operators 
$T_l:(N,\tau)\to (N,\tau)$ 
is simultaneously absolutely dilatable if there exist a tracial von Neumann algebra $(N',\tau')$, 
a commuting $n$-tuple $(U_1,\dots,U_n)$ of trace preserving $\star$-automorphisms on $N'$ and a unital 
one-to-one trace preserving and
$w^*$-continuous $\star$-homomorphism $J: N\to N'$ such that 
\begin{align*}
T_1^{k_1}\cdots T_n^{k_n}=
\mathbb{E}U_1^{k_1}\cdots U_n^{k_n}J
\end{align*}
for all $k_l\geq 0$, $1\leq l\leq n$, where 
$\mathbb{E}:N'\to N$ is  the conditional 
expectation associated with $J$.

\item[(b)]
Let $1\leq p <+\infty$ and 
let $(N,\tau)$ be a tracial von Neumann algebra. We
say that a commuting family $(T_1,\ldots,T_n)$ of 
operators $T_l:L^p(N)\to L^p(N)$
is simultaneously completely $p$-dilatable, if there exist a tracial von Neumann algebra $(N',\tau)$, a
commuting $n$-tuple $(U_1,\dots,U_n)$ of 
invertible complete isometries on $L^p(N')$
and two complete contractions $J: L^p(N)\to L^p(N')$ and $Q:L^p(N')\to L^p(N)$ such that 
\begin{align*}
T_1^{k_1}\cdots T_n^{k_n}=
QU_1^{k_1}\cdots U_n^{k_n}J
\qquad \hbox{on}\ L^p(N)
\end{align*}
for all $k_l\geq 0$, $1\leq l\leq n$.
\end{itemize}
\end{definition}

As in the case of one operator (see Lemma \ref{lem absdila implique dila p}), we see
that a simultaneously absolutely dilatable family is simultaneously completely $p$-dilatable.

\begin{theorem}\label{th principal multivariables} 
Let $I$ be an index set and for any
$1\leq l \leq n$, let $M^l=(m_{ij}^l)_{i,j\in I}\in \M_I$  
such that $T_{M_l}\colon B(l^2(I))\to B(l^2(I))$ is a 
unital positive bounded Schur multiplier.
The following assertions are equivalent:
\begin{enumerate}
\item the family $(T_{M^l})_{l=1}^n$ is simultaneously absolutely dilatable;
\item each $T_{M^l}$ is absolutely dilatable.
\end{enumerate}
\end{theorem}

\begin{proof}
Implication $1\Rightarrow 2$ is clear. Conversely, we suppose that each $T_{M^l}$ is absolutely dilatable and we use Theorem \ref{equivalence l2}. Thus for all $1\leq l\leq n$, there exist  a tracial von Neumann algebra $(N^l,\tau^l)$,  unitaries of $N$ $(u^l_j)$ such that for all $i,j\in I$:
$$
m^l_{ij}=\tau^l\left( \left(u^l_i\right)^*u^l_j\right).
$$

We denote the infinite von Neumann tensor product $\overline{\otimes}_\Z N$ by $N^\infty$ and we 
let $\tau^\infty$ denote the normal faithful finite trace on $N^\infty$ 
(see \cite{Takesaki3}). We denote the unit of $N^\infty$ by $1_\infty$.

We let $(N,\tau_N)=(N^1\overline{\otimes}\cdots \overline{\otimes}N^n,\tau^1\overline{\otimes}\cdots\overline{\otimes} \tau^n)$ and  $N'=B(l^2(I))\overline{\otimes} N^\infty$. We consider the standard n.s.f trace $\tau_{N'}:=\tr_I\overline{\otimes} \tau_N^\infty$ on $N'$. We set for all $j\in I$, the map $e_{jj}: l^2(I)\to  l^2(I)$ such that for all $x=(x_i)_{i\in I}\in l^2(I)$, $e_{jj}(x)=(\dots,0,x_j,0,\dots)$ where $x_j$ is in the $j$ position.

We let for all $1\leq l \leq n$,
\begin{align*}
    \begin{array}{ccc}
u^l =\sum_{j\in I} e_{jj}\otimes\cdots1_N\otimes 1_N\otimes &\underbrace{1_{N^1}\otimes \cdots\otimes 1_{N^{l-1}} \otimes u^l_j\otimes  1_{N^{l+1}}\otimes \cdots \otimes 1_{N^n}}&\otimes 1_N\cdots\\
&0&
\end{array}
\end{align*}
where the summation is taken in $w^*$-topology of $N'$.

The element $u^l$ is unitary because all $u^l_j$ are unitaries. We introduce the right shift $S : N^\infty\to N^\infty$. This is
a normal, trace preserving $\star$-automorphism such that 
for all $(x_n)_{n\in \Z}\subset N$,
\begin{align*}
S(\cdots\otimes x_0\otimes x_1\otimes x_2\otimes \cdots)=\cdots\otimes x_{-1}\otimes x_0\otimes x_1\otimes \cdots.
\end{align*} we let the following mappings:
\begin{align*}
&J:B(l^2(I))\to N' \quad x\mapsto x\otimes 1_{\infty}\\
&U_l:N'\to N'\quad y\mapsto (u^l)^*((Id\otimes S)(y))u^l.
\end{align*}

We remark that $\E =id\otimes \tau_N^\infty$ is the normal faithful canonical conditional 
expectation associated with $J$ and it preserves the trace. In addition $J$ is a $\star$-homomorphism which is trace preserving and the  $U_l$ are $\star$-isomorphisms which are trace preserving. All $U_l$  are commuting to each other. With the same computation in \cite[Proof of Theorem 1.9]{SKR}, we obtain $T_{M_1}^{k_1}\cdots T_{M_n}^{k_n}=\mathbb{E}U_1^{k_1}\cdots U_n^{k_n}J$, for all $k_l\in \N_0$, $1\leq l\leq n$.
\end{proof}

\vskip 1cm

\begin{large}
Acknowledgments:\end{large} I would like to thank Christian Le Merdy, my thesis supervisor for all his support and his help.
The LmB receives support from
the EIPHI Graduate School (contract ANR-17-EURE-0002)
and the author was supported by the ANR project Noncommutative analysis on groups and quantum groups (No./ANR-19-CE40-0002).

\vskip 1cm
\bibliographystyle{abbrv}

\vspace{0.8cm}

\end{document}